\documentclass[12pt,a4paper]{article}
\usepackage[T1]{fontenc}
\usepackage[centertags]{amsmath}
\usepackage{amsfonts}
\usepackage{amssymb}
\usepackage{amsthm}
\usepackage{newlfont}
\usepackage{indentfirst}
\usepackage{float}
\usepackage{enumitem}
\usepackage{algorithm}
\usepackage{algorithmicx}
\usepackage{algpseudocodex}

\DeclareMathOperator{\tr}{tr}

\DeclareMathOperator{\GL}{GL}

\DeclareMathOperator{\sgn}{sgn}
\DeclareMathOperator{\Aut}{Aut}
\DeclareMathOperator{\Inn}{Inn}
\DeclareMathOperator{\Out}{Out}
\DeclareMathOperator{\PPsi}{\Psi}

\newtheorem{thm}{Theorem}[section]
\newtheorem{cor}[thm]{Corollary}
\newtheorem{lem}[thm]{Lemma}
\newtheorem{prop}[thm]{Proposition}
\newtheorem{conj}[thm]{Conjecture}

\theoremstyle{definition}
\newtheorem{definition}[thm]{Definition}
\newtheorem{rem}[thm]{Remark}
\newtheorem{ex}[thm]{Example}

\algrenewcommand\algorithmicrequire{\textbf{Input:}}
\algrenewcommand\algorithmicensure{\textbf{Output:}}

\newcommand{\blank}{{\mspace{1mu}\cdot\mspace{1mu}}}
\newcommand{\N}[1]{\lVert #1 \rVert}
\newcommand{\Sr}{\Sigma^{*}_{\mathrm{red}}}
\newcommand{\Scr}{\Sigma^{*}_{\mathrm{c\,red}}}

\title{Translation length formula for two-generated groups acting on trees}

\author{Kamil Orzechowski\\
University of Rzesz\'ow\\
Rejtana 16C Street\\
35-959 Rzesz\'ow, Poland\\ 
\texttt{kamilo@dokt.ur.edu.pl}}

\date{February 24, 2026}

\begin{document}

\maketitle

\begin{abstract}
    We investigate translation length functions for two-generated groups acting by isometries on $\Lambda$-trees, where $\Lambda$ is a totally ordered abelian group.
    In this context, we provide an explicit formula for the translation length of any element of the group, under certain assumptions on the translation lengths of its generators and their products. Our approach is purely combinatorial and uses only the defining axioms of pseudo-lengths. As shown by Parry, pseudo-lengths coincide with the translation length functions for actions on $\Lambda$-trees. Furthermore, we prove that, under certain conditions on four elements $\alpha, \beta, \gamma, \delta \in \Lambda$, there exists a unique pseudo-length on the free group $F(a,b)$ assigning these values to $a$, $b$, $ab$, $ab^{-1}$, respectively.
    
    Applications include results on properly discontinuous actions and discrete free groups of isometries. We also develop an algorithmic approach to studying translation length functions arising from free actions on $\mathbb{R}$-trees. Based on this, we state a conjecture that would lead to a description of $\mathrm{Aut}{(F_2)}$-orbits in the Culler--Vogtmann outer space.


\end{abstract}

\section{Introduction}

The concept of translation length appears in the theory of groups acting on trees by isometries. Most generally, this notion is applied to $\Lambda$-trees, for any totally ordered nontrivial abelian group $\Lambda$. Morgan and Shalen \cite[Ch. II.1]{Morgan-Shalen} were probably the first to define $\Lambda$-trees in this generality. The theory was further developed by Alperin and Bass \cite{Alperin-Bass}. We refer to \cite{Chiswell} for a thorough introduction to $\Lambda$-trees.

Culler and Morgan listed some algebraic properties of the translation length function $G\ni g\mapsto \N{g}$ for an action of a group $G$ on an $\mathbb{R}$-tree by isometries \cite[1.11]{Culler-Morgan}. They called any function from $G$ to $[0,\infty)$ satisfying these properties a {\em pseudo-length}, and asked if every pseudo-length is the translation length function of some action of $G$ on an $\mathbb{R}$-tree. This question was answered affirmatively by Parry \cite{Parry}, who worked in the general context of $\Lambda$-trees.

It is known that if two isometries $a,b$ of a $\Lambda$-tree $X$ satisfy the conditions
\begin{equation}\label{eq:introduction_conditions}
    \N{a}>0, \quad \N{b}>0, \quad \lvert \N{a}-\N{b}\rvert < \min \{\N{ab},\N{ab^{-1}}\},
\end{equation}
the group they generate acts freely, properly discontinuously, and without inversions on $X$; see \cite[Propositions 1 and 2]{Chiswell-article} or \cite[Lemmas 3.3.6 and 3.3.8]{Chiswell}. The cited proofs are geometric in nature and rely on drawing pictures or ``ping-pong'' type arguments.

We present a combinatorial approach, using only the defining conditions of a pseudo-length and not referring to any geometric interpretation. Moreover, we obtain an explicit formula for the translation length $\N{g}$ for $g\in \left<a,b\right>$ if the conditions \eqref{eq:introduction_conditions} are satisfied.

\bigskip

\noindent\textbf{Main results}

\begin{enumerate}[label={\bf\arabic*.}]
    \item {\it If $\N{\blank}$ is a pseudo-length on a group $G$ and $a,b\in G$ satisfy \eqref{eq:introduction_conditions},
    then 
    \begin{equation}\label{eq:introduction_formula}
        2\N{w}= \left(\sum_{i=1}^{n-1}\N{x_i x_{i+1}}\right) + \N{x_n x_1}>0,
    \end{equation}
     for any cyclically reduced word ${w=x_1\dots x_n}$, $x_i\in \{a,b,a^{-1},b^{-1}\}$, $n\ge 1$ (Theorem \ref{th:main}).
   }

   \item {\it If $\alpha,\beta,\gamma,\delta\in \Lambda$ satisfy the conditions:
   \begin{gather}
    \alpha>0, \;\beta>0,\; \lvert\alpha-\beta\rvert < \min\{\gamma,\delta\};\label{eq:intro_1}\\
    \text{either } \gamma=\delta > \alpha +\beta  \quad \text{or } \max\{\gamma, \delta\}=\alpha+\beta;\label{eq:intro_2}\\
    \gamma - \alpha - \beta \in 2\Lambda,\quad \delta - \alpha - \beta \in 2\Lambda;\label{eq:intro_3}
\end{gather}
   then there exists a unique pseudo-length on the free group $F(a,b)$ such that $\N{a}=\alpha$, $\N{b}=\beta$, $\N{ab}=\gamma$, and $\N{ab^{-1}}=\delta$ (Theorem \ref{thm:unique}).}
\end{enumerate}

The uniqueness part of our second result follows from \eqref{eq:introduction_formula}; note that \eqref{eq:intro_1} is the counterpart of \eqref{eq:introduction_conditions}. The existence part (Proposition \ref{prop:existence}) is nontrivial. We define a function on $F(a,b)$ by a formula that mimics \eqref{eq:introduction_formula}. The conditions \eqref{eq:intro_2} and \eqref{eq:intro_3}, which may seem technical, are used to prove that the constructed function satisfies the axioms of a pseudo-length. Notice that \eqref{eq:intro_3} is only necessary if $2\Lambda \ne \Lambda$.

\bigskip

In the final section, we provide some applications of our results. We draw a conclusion about discrete and free subgroups acting on trees (as in \cite{Conder} and \cite{Conder2022}). Explicitly, we prove the following corollary, which generalizes Conder's geometric result \cite[Corollary 3.6]{Conder} (formulated for a continuous action on a $\mathbb{Z}$-tree):

{\it If $G$ is a topological group acting on a $\Lambda$-tree $(X,d)$ by isometries in such a way that for some $x_0\in X$ the map ${G\ni g\mapsto gx_0 \in X}$ is continuous and $a,b\in G$ satisfy \eqref{eq:introduction_conditions}, where $\N{\blank}\colon G\to \Lambda_{+}$ is the translation length function of this action, then the subgroup $\left<a, b\right>$ is free of rank two and discrete (Corollary \ref{cor:Conder}).}

We also construct an algorithm that outputs a basis $(g,h)$ of $F(a,b)$ satisfying \eqref{eq:introduction_conditions} (with $g,h$ in the place of $a,b$), when given any pseudo-length $\N{\blank}\colon F(a,b)\to \Lambda_{+}$  that is {\em purely hyperbolic} (i.e., $\N{g}>0$ for all $g\ne 1$) and takes values in a nontrivial subgroup of $\mathbb{R}$ (Algorithm 1).

Based on Algorithm 1, we study $\Aut{(F_2)}$-orbits in the space $\PPsi{(F_2)}$ of all purely hyperbolic real-valued pseudo-lengths on $F(a,b)$ (Theorem \ref{thm:description_of_orbits} and Remark \ref{rem:final_rem}). As a result, we find a set whose intersection with any such orbit is nonempty. This set is defined as the union of two subsets \eqref{eq:Y1_and_Y2}, which are parametrized by four real numbers. We conjecture that it contains exactly one member of every $\Aut{(F_2)}$-orbit in $\PPsi{(F_2)}$ (Conjecture \ref{conj:conjecture}).

If our conjecture is true, we would obtain a description of $\Aut{(F_2)}$-orbits in the space of all pseudo-lengths on $F(a,b)$ that are the translation length functions of free actions on $\mathbb{R}$-trees. The space of all such functions is related to the concept of the Culler--Vogtmann {\em outer space} \cite{Culler-Vogtmann}, which (for rank two) can be identified with the projectivization of $\PPsi{(F_2)}$.

\section{Preliminaries}

In this paper, unless otherwise specified, let $\Lambda$ be a fixed totally (linearly) ordered nontrivial abelian group. We will write $\Lambda$ additively. Let also $\Lambda_{+}:=\{\lambda\in \Lambda\colon \lambda \ge 0\}$, and $\lvert \lambda\rvert:=\max\{\lambda,-\lambda\}$ for $\lambda\in \Lambda$. Without mentioning it, we will use the fact that multiplying both sides of any inequality or equality between two elements of $\Lambda$ by $2$ yields an equivalent one. When we write $\lambda'=\frac12 \lambda$ for $\lambda\in \Lambda$, we mean $2\lambda' = \lambda$, implicitly assuming that such a (necessarily unique) $\lambda'\in\Lambda$ exists. For the basic theory of ordered abelian groups, see \cite[Ch. 1, \S1]{Chiswell}.

We call $\Lambda$ {\em Archimedean} if, given $a,b\in \Lambda$ with $b\ne 0$, there exists $n\in \mathbb{Z}$ such that $a< nb$. It is known that $\Lambda$ is Archimedean if and only if there exists an embedding of ordered abelian groups $\Lambda\to \mathbb{R}$ \cite[Theorem 1.1.2]{Chiswell}. An example of a non-Archimedean totally ordered abelian group is $\mathbb{Z}\oplus \mathbb{Z}$ with the lexicographic ordering.

A {\em $\Lambda$-metric space} $(X,d)$ is defined using the same axioms as for a metric space, except that $d$ takes values in $\Lambda$ instead of $\mathbb{R}$ \cite[Ch. 1, \S2]{Chiswell}.
As usual, $d$ induces a topology on $X$ with the family of open balls as its basis. An example of a $\Lambda$-metric space is $\Lambda$ itself with the $\Lambda$-metric $d(x,y):=\lvert x-y\rvert$ for $x,y\in \Lambda$. It is the only $\Lambda$-metric that we will consider in $\Lambda$; the topology induced by $d$ turns $\Lambda$ into a topological group.

An {\em isometry} of $\Lambda$-metric spaces is defined as for metric spaces; we require that an isometry be surjective, otherwise we use the term {\em isometric embedding}.  A {\em segment} in a $\Lambda$-metric space $(X,d$) is the image of an isometric embedding $i\colon [a,b]_{\Lambda}\to X$ of a closed interval $[a,b]_{\Lambda}:=\{x\in \Lambda\colon a\le x\le b\}$, $a,b\in \Lambda$, $a\le b$, into $X$; $i(a)$ and $i(b)$ are then called the {\em endpoints} of the segment. A $\Lambda$-metric space $(X,d)$ is called {\em geodesic} if for any $x,y\in X$ there exists a segment in $X$ with $x$ and $y$ as its endpoints.
\begin{definition}[{\cite[Ch. 2, \S1]{Chiswell}}]\label{def:L-tree}
    A {\em $\Lambda$-tree} is a geodesic $\Lambda$-metric space $(X,d)$ such that:
    \begin{enumerate}
        \item if two segments of $(X,d)$ intersect in a single point, which is an endpoint of both, then their union is a segment;
        \item the intersection of two segments with a common endpoint is also a segment.
    \end{enumerate}
\end{definition}

It follows from the above definition that, in a $\Lambda$-tree $(X,d)$, there is a unique segment with $x,y$ as its endpoints for every $x,y\in X$ \cite[Lemma 2.1.1]{Chiswell}, let us denote it by $[x,y]$. If $A \subseteq X$ is such that $[x,y]\subseteq A$ whenever $x,y\in A$, we call $A$ a {\em subtree} of $X$.
A trivial example of a $\Lambda$-tree is $\Lambda$ itself, then segments in $\Lambda$ are exactly the closed intervals.

\bigskip

The most important examples of $\Lambda$-trees are $\mathbb{Z}$-trees, which are just ``ordinary'', graph-theoretic trees with the shortest path metric on the set of vertices, and $\mathbb{R}$-trees, which can be characterized as uniquely arcwise-connected geodesic metric spaces \cite[Proposition 2.2.3]{Chiswell}.

Any $\mathbb{Z}$-tree $X$ can be isometrically embedded in a canonical way into an $\mathbb{R}$-tree $\mathrm{real}(X)$, called the {\em geometric realization} of $X$ \cite[Theorem II.1.9]{Morgan-Shalen}. It is obtained by taking an isometric copy of $[0,1]_{\mathbb{R}}$ for each of the edges of $X$, doing the appropriate identification of endpoints, and extending the metric suitably.
By a {\em simplicial tree} we mean any $\mathbb{R}$-tree that is homeomorphic to $\mathrm{real}(X)$ for some $\mathbb{Z}$-tree $X$. This terminology is consistent with that of \cite{Culler-Vogtmann}; Chiswell \cite{Chiswell} uses the term {\em polyhedral tree} instead.

\bigskip

All isometries of a $\Lambda$-tree $(X,d)$ onto itself can be divided into three types: elliptic, hyperbolic and inversions \cite[Ch. 3, \S1]{Chiswell}. Let $g$ be an isometry of a $\Lambda$-tree $(X,d)$. It is called {\em elliptic} if it has a fixed point in $X$; $g$ is called an {\em inversion} if $g$ has no fixed points in $X$ but $g^2$ does; otherwise $g$ is called {\em hyperbolic}. The {\em translation length} of $g$ \cite[p. 297]{Parry} is defined as
\begin{equation}\label{eq:t_l}
    \N{g} := \begin{cases}
        0 & \text{if } g \text{ is an inversion},\\
        \min \{d(x, gx) \colon x\in X\}& \text{otherwise.}
    \end{cases}
\end{equation}
In fact, if $g$ is not an inversion, the set of points for which the minimum in \eqref{eq:t_l} is reached is a  nonempty closed subtree of $X$. In the case when $\Lambda = 2\Lambda$ (e.g., $\Lambda=\mathbb{R})$, there are no inversions.
Hyperbolic isometries are precisely those with $\N{g}>0$. If $g$ is hyperbolic, then the set $\{x\in X \colon d(x,gx)=\N{g}\}$ is called the {\em axis} of $g$; it is isometric to a subtree of $\Lambda$ and the action of $g$ on its axis corresponds to the translation by $\N{g}$, which justifies the terminology \cite[cf. Theorem 3.1.4 and Corollary 3.1.5]{Chiswell}.

\begin{ex}
Let $F$ be a field with a {\em valuation} $v$, i.e., a group homomorphism $v\colon F^{*} \to \Lambda$ satisfying $v(a+b)\ge \min\{v(a),v(b)\}$ for $a,b\in F$ (with the convention $v(0)=\infty)$. There exists a $\Lambda$-tree $X_{v}$ on which $\GL{(2,F)}$ acts by isometries \cite[Appendix A]{Alperin-Bass} with the translation length
\[\N{g}=\max\{v(\det{g})-2v(\tr{g}),0\} \quad \text{for } g\in \GL{(2,F)}.\]
The $\Lambda$-tree $X_{v}$ generalizes the Bruhat--Tits tree, which was constructed for a discrete valuation $v$, see \cite[Ch. II]{Serre}.
\end{ex}

Let us discuss pseudo-lengths and their connection with translation length functions.

\begin{definition}[{\cite[(0.2)]{Parry}}]
    Let $G$ be a group. A function $\N{\blank}\colon G \to \Lambda_{+}$ is called a {\em pseudo-length} if it satisfies the following conditions, called axioms:
    \begin{enumerate}[label=(A\arabic*), start=0]
        \item $\max\{0, \N{gh}-\N{g}-\N{h}\}\in 2\Lambda$ for any $g,h\in G$ with $\N{g}>0$, $\N{h}>0$;
        \item $\N{ghg^{-1}}=\N{h}$ for any $g,h\in G$;
        \item $\N{gh}=\N{gh^{-1}}$ or $\max\{\N{gh},\N{gh^{-1}}\}\le \N{g} + \N{h}$, for any $g,h\in G$;
        \item $\N{gh}=\N{gh^{-1}}>\N{g} + \N{h}$ or $\max\{\N{gh},\N{gh^{-1}}\}= \N{g} + \N{h}$,\\ for any $g,h\in G$ with $\N{g}>0$, $\N{h}>0$.
    \end{enumerate}
\end{definition}

According to \cite[Main Theorem, p. 298]{Parry}, $\N{\blank}\colon G\to \Lambda_{+}$ is a pseudo-length on a group $G$ if and only if there exists a $\Lambda$-tree $X$ and an action of $G$ on $X$ by isometries such that $\N{\blank}$ is the translation length function for this action.

\begin{rem}
    The original definition of a (real-valued) pseudo-length given by Culler and Morgan \cite[1.11]{Culler-Morgan} included the following additional axioms:
    \begin{enumerate}[label=(A\arabic*), start=4]
    \item $\N{1}=0$;
    \item $\N{g^{-1}}=\N{g}$ for any $g\in G$.
    \end{enumerate}

    Parry pointed out that they are redundant. Indeed, if $\N{1}>0$, the application of (A3) to $g=h=1$ would lead to a contradiction; (A5) follows from applying (A2) twice: once for the pair $(1,g)$ and once for $(1,g^{-1})$.
\end{rem}

\begin{definition}
    Let $\N{\blank}\colon G\to \Lambda_{+}$ be a pseudo-length on a group $G$. We call $\N{\blank}$ {\em purely hyperbolic} if $\N{g}>0$ for all $g\in G\setminus\{1\}$.
    %  By extension we will also call $G$ a {\em purely hyperbolic group} if the pseudo-length is known from the context.
    % We call $(g,h)\in G\times G$ a {\em ping-pong pair} if $\N{g}>0$, $\N{h}>0$ and $\lvert \N{g}-\N{h}\rvert < \min \{\N{gh},\N{gh^{-1}}\}$.
\end{definition}

We say that an action of a group $G$ on a set $X$ is {\em free} if every point of $X$ has trivial stabilizer. If $X$ is endowed with a topology, the action is called {\em properly discontinuous} if every $x\in X$ has a neighborhood $U$ such that $gU\cap U \neq \emptyset$ implies $g=1$.

It follows from \eqref{eq:t_l} that an action of $G$ on a $\Lambda$-tree $X$ is free and without inversions if and only if the associated translation length function is purely hyperbolic.

\bigskip

Let us introduce the concept of a {\em ping-pong} pair in a group $G$ equipped with a pseudo-length. We named it so because the defining condition often appears in the literature in so called ping-pong type arguments.

\begin{definition}
    Let $\N{\blank}\colon G\to \Lambda_{+}$ be a pseudo-length on a group $G$. We call $(g,h)\in G\times G$ a {\em ping-pong pair} if $\N{g}>0$, $\N{h}>0$, and {$\lvert \N{g}-\N{h}\rvert < \min \{\N{gh},\N{gh^{-1}}\}$}.
\end{definition}

Geometrically, the fact that $(g,h)$ is a ping-pong pair corresponds to the situation when $g, h$ are hyperbolic isometries of a $\Lambda$-tree such that either the axes of $g,h$ are disjoint or their intersection is a segment of length less then $\min\{\N{g},\N{h}\}$ (see, for example, \cite[Lemma 3.1]{Conder2022} and the paragraph that follows it).

\section{Results}

Let us collect some useful properties of every pseudo-length.

\begin{lem}
    Let $\N{\blank}$ be a pseudo-length on a group $G$. For any $g,h\in G$ we have:
    \begin{gather}
        \N{g^n}=\lvert n\rvert \N{g} \quad \text{for } n\in\mathbb{Z};\label{eq:g^n}\\
        \textrm{if} \quad \N{gh^{-1}}>\N{g}+\N{h}, \quad \textrm{then } \N{gh}=\N{gh^{-1}};\label{eq:lem_first}\\
        \textrm{if} \quad \N{gh^{-1}}<\N{g}+\N{h}, \,\N{g}>0, \,\N{h}>0, \quad \textrm{then } \N{gh}=\N{g}+\N{h}.\label{eq:lem_second}
    \end{gather}
\end{lem}

\begin{proof}
    The formula \eqref{eq:g^n} was proved in \cite[Lemma 6.1]{Culler-Morgan}. The implications \eqref{eq:lem_first} and \eqref{eq:lem_second} follow from (A2) and (A3) respectively.
\end{proof}

We now establish some terminology and notation that we will adhere to throughout the paper.
Let $a,b,a^{-1},b^{-1}$ be distinct symbols and $\Sigma:=\{a,b,a^{-1},b^{-1}\}$. Extend the mapping $a\mapsto a^{-1}$, $b\mapsto b^{-1}$ to an involution $^{-1}\colon \Sigma \to \Sigma$.
Let $\Sigma^{*}$ be the free monoid of all words over the alphabet $\Sigma$ with the operation of concatenation (written simply $uv$ for $u,v\in \Sigma^{*}$) and the identity $1$ (the empty word). 
We put $w^{-1}:=x_{n}^{-1}\dots x_{1}^{-1}$ for any $w=x_1\dots x_n$, where $x_i\in \Sigma$ $(1\le i\le n)$; also $1^{-1}:=1$.

A word $w\in\Sigma^{*}$ is {\em reduced} if $w=1$ or $w=x_1\dots x_n$, where $x_i\in \Sigma$ $(1\le i\le n)$ and $x_{i}\ne x_{i+1}^{-1}$ $(1\le i <n)$. We denote by $\Sr$ the set of all reduced words over $\Sigma$; it is in a canonical bijection with $F(a,b)$, the free group over $\{a,b\}$.

A word $w\in\Sr$ is {\em cyclically reduced} if $w=1$ or $w=x_1\dots x_n$, where $x_i\in \Sigma$ $(1\le i\le n)$ and $x_{n}\ne x_{1}^{-1}$. We denote by $\Scr$ the set of all cyclically reduced words over $\Sigma$.
By a {\em cyclic shift} we mean a transformation of the form
$x_1 \dots x_n \mapsto x_k x_{k+1} \dots x_n x_1 \dots x_{k-1}$ for some $1\le k\le n$, where $x_i\in \Sigma$ $(1\le i\le n)$.
Recall that every $g\in F(a,b)$ is conjugated to an element of $F(a,b)$ represented by a cyclically reduced word, which is unique up to a cyclic shift.

Let $w\in\Scr\setminus\{1\}$. We say that $u\in \Sigma^{*}\setminus\{1\}$ is a {\em repeating subword} of $w$ if there exists a word $w'$ arising from $w$ via a cyclic shift and satisfying
\[w' = v_1 u v_2 u \quad \text{for some} \quad v_1, v_2 \in \Sigma^{*}.\]
Note that $u,v_1,v_2$ are necessarily reduced.

Assume that $G$ is a group, $a,b\in G$, and $w\in \Sigma^*$. When it is clear from the context (in particular, if $w$ appears inside the pseudo-length symbol), we will simply write $w$ for the image of $w$ under the evaluation homomorphism from $\Sigma^*$ to $G$. If we want to emphasize that $w$ should be treated as a formal word, we will explicitly write $w\in \Sigma^*$.

\bigskip

The following technical lemma will be used a couple of times in the paper.

\begin{lem}\label{lem:technical}
   Assume that a function $f\colon \Sigma \times \Sigma \to \Lambda$ satisfies the following conditions:
    \begin{enumerate}
        \item $f(x,y)=f(y,x)=f(y^{-1},x^{-1})$, $f(x,x^{-1})=0$ for all $x,y\in \Sigma$;
        \item either \quad ${2f(a,b)=2f(a,b^{-1})>f(a,a)+f(b,b)}$\\
        or \quad ${2\max\{f(a,b),f(a,b^{-1})\}=f(a,a)+f(b,b)}$;
        \item $f(a,a)>0$, $f(b,b)>0$, $\lvert f(a,a)-f(b,b) \rvert < 2\min \{f(a,b),f(a,b^{-1})\}$.
    \end{enumerate}
    Then 
    \begin{equation}\label{eq:technical}
        f(y^{-1},z)\le f(x,y) + f(x,z) \quad \text{for all } x,y,z\in \Sigma,
    \end{equation}
    and the inequality \eqref{eq:technical} is strict if and only if $y\ne x^{-1}$ and $z\ne x^{-1}$.
\end{lem}

\begin{proof}
    First, notice that it follows from (i) that each of the conditions (ii) and (iii) remains true if we replace one or both of $a,b$ by its inverse.

    If $y$ or $z$ equals $x^{-1}$, then the equality in \eqref{eq:technical} follows from (i). Assume that $y\ne x^{-1}\ne z$. If $y=z$, we get by (i) and (iii) that $0<2f(x,y)$, as desired. Suppose that $y\ne z$. It follows from the assumptions on $x,y,z$ that either $y=z^{-1}$ or one of $y,z$ equals $x$.

    If $y=z^{-1}$, we have to prove that 
    \begin{equation}\label{eq:technical_goal_1}
        f(z,z) < f(x,z^{-1}) + f(x,z).
    \end{equation}
    Note that $x,z$ are powers of distinct letters in $\{a,b\}$. If the first alternative in (ii) holds, then $2f(x,z^{-1}) + 2f(x,z)>2f(x,x)+2f(z,z)>2f(z,z)$, hence \eqref{eq:technical_goal_1}. If the second alternative in (ii) holds, then $2f(x,z^{-1})+2f(x,z)=f(x,x)+f(z,z)+2\min\{f(x,z),f(x,z^{-1})\}$ and \eqref{eq:technical_goal_1} is equivalent to $2f(z,z)<f(x,x)+f(z,z)+2\min\{f(x,z),f(x,z^{-1})\}$, which follows from (iii).

    Assume now that $y=x$ (the case $z=x$ is similar). Note that $x,z$ are again powers of distinct letters in $\{a,b\}$. Our goal is to prove that
    \begin{equation}\label{eq:technical_goal_2}
        f(x^{-1},z)-f(x,z)<f(x,x),
    \end{equation}
    which is obviously true if $f(x^{-1},z)\le f(x,z)$. If $f(x^{-1},z)>f(x,z)$, then by (i) and (ii) we have $2f(x^{-1},z)=2\max\{f(x,z),f(x,z^{-1})\}=f(x,x)+f(z,z)$ and \eqref{eq:technical_goal_2} is equivalent to $f(x,x)+f(z,z)-2\min\{f(x,z),f(x,z^{-1})\}<2f(x,x)$, which follows from (iii).
\end{proof}

Let us prove a lemma that will be used several times in the proof of Theorem \ref{th:main}.

\begin{lem}\label{lem:yz}
    Let $\N{\blank}$ be a pseudo-length on the group $F(a,b)$. If $(a,b)$ is a ping-pong pair, then 
    \begin{equation}\label{eq:yz}
        \N{y^{-1}z} < \N{xy} + \N{xz}
    \end{equation}
    for all $x,y,z\in \Sigma$ satisfying $y\ne x^{-1}$ and $z\ne x^{-1}$.
\end{lem}

\begin{proof}
    Define $f\colon \Sigma \times \Sigma \to \Lambda$ by putting $f(x,y):=\N{xy}$ for $x,y\in \Sigma$. Let us show that $f$ satisfies the assumptions of Lemma \ref{lem:technical}.

    Indeed, (i) follows from (A1), (A4) and (A5). Since by \eqref{eq:g^n} we have $f(a,a)=2\N{a}>0$ and $f(b,b)=2\N{b}>0$, we obtain (ii) from (A3), and (iii) from the assumption on $a,b$.
    Therefore, the application of Lemma \ref{lem:technical} to $f$ yields \eqref{eq:yz}.
\end{proof}

\begin{thm}\label{th:main}
    Let $\N{\blank}$ be a pseudo-length on a group $G$ and $a,b\in G$.
    If $(a,b)$ is a ping-pong pair, then for any cyclically reduced word ${w=x_1\dots x_n}$, $x_i\in \Sigma$ $(1\le i\le n)$, $n\ge 1$, the following holds:
    \begin{equation}\label{eq:main}
        2\N{w}= \left(\sum_{i=1}^{n-1}\N{x_i x_{i+1}}\right) + \N{x_n x_1}>0.
    \end{equation}
\end{thm}

\begin{proof}
    Identify $F(a,b)$ with $\Sr$ canonically.
    We can see that the function $F(a,b)\ni w \mapsto \N{w}$ is a pseudo-length on $F(a,b)$, and if we prove the claim for $F(a,b)$ with this pseudo-length, the claim for $G$ with $\N{\blank}$ will follow. Therefore, we assume that $G=F(a,b)$ for the rest of the proof. Note also that, by (A1) and (A5), both sides of \eqref{eq:main} do not change if we replace $w$ by $w^{-1}$ or a cyclic shift of $w$. Moreover, $\N{a^{\pm 1}b^{\pm 1}}>0$, $\N{a^{\pm2}}>0$ and $\N{b^{\pm 2}}>0$ by the assumption on $a,b$, (A1), (A5) and \eqref{eq:g^n}; so the right-hand side of \eqref{eq:main} is always positive.

    We will prove \eqref{eq:main} by induction on $n$. Since $2\N{g}=\N{g^2}$ by \eqref{eq:g^n} and $2\N{gh}=\N{gh}+\N{hg}$ by (A1) for all $g,h\in G$, the equation \eqref{eq:main} is true for $n\in \{1,2\}$. Assume that $n>2$ and \eqref{eq:main} is true for all $w\in \Scr$ of length less than $n$. We distinguish two cases depending on whether $w$ contains a repeating subword.

    \medskip

    Case 1: $w$ does not contain a repeating subword. Then $n\le 4$ and $n\neq 3$ (since $w$ is cyclically reduced). Thus, up to a cyclic shift, $w$ is one of the commutators $[a,b]$, $[a,b^{-1}]$. Since $[a,b^{-1}]^{-1}=[b^{-1},a]$ is a cyclic shift of $[a,b]$, we only need to show \eqref{eq:main} for $w=[a,b]=aba^{-1}b^{-1}$.
    Applying \eqref{eq:yz} for $x:=b$, $y:=a^{-1}$, $z:=a$, we obtain
    \[\N{ab(a^{-1}b)^{-1}}=\N{a^2}<\N{ba^{-1}}+\N{ba}=\N{ab}+\N{a^{-1}b}.\] Then, from \eqref{eq:lem_second} with $g:=ab$, $h:=a^{-1}b$, we deduce that 
    \[
    \N{aba^{-1}b}=\N{ab}+\N{a^{-1}b}.
    \]
    By \eqref{eq:yz} with $x:=a$, $y:=b^{-1}$, $z=b$, we also have
    $\N{ab^{-1}}+\N{ab}>\N{b^2}$; so by (A1), (A5) and \eqref{eq:g^n} we obtain 
    \[\N{aba^{-1}b}=\N{ab}+\N{a^{-1}b}>2\N{b}=\N{aba^{-1}}+\N{b^{-1}}.\]
    The application of \eqref{eq:lem_first} for $g:=aba^{-1}$ and $h:=b^{-1}$ yields $\N{aba^{-1}b^{-1}}=\N{aba^{-1}b}=\N{ab}+\N{ab^{-1}}$. Hence,
    \[2\N{aba^{-1}b^{-1}}=2\N{ab}+2\N{ab^{-1}}=\N{ab}+\N{ba^{-1}}+\N{a^{-1}b^{-1}}+\N{b^{-1}a},\]
    which proves \eqref{eq:main} for $w=aba^{-1}b^{-1}$.

    \medskip
    
    Case 2: $w$ contains a repeating subword $u=x_1\dots x_k$, $x_i\in \Sigma$ ${(1\le i\le k)}$, of maximal length $k\ge 1$. We may assume that $w=v_1 u v_2 u$ for some ${v_1,v_2\in \Sr}$. Let us consider four subcases.

    \smallskip

    Subcase 2a: $v_1\ne 1$ and $v_2\ne 1$. Let us write  $v_1=y_1 \dots y_m$, $v_2=z_1\dots z_s$, $m,s\ge 1$, as reduced words over $\Sigma$. Notice that, by the maximality of $u$, we have $y_1\ne z_1$ and $y_m\ne z_{s}$. Let us show that $g:=v_1u$ and $h:=v_2u$ satisfy the assumptions of \eqref{eq:lem_second}.
    Using the induction hypothesis, we calculate
    \begin{equation*}
        \begin{split}
        2\N{gh^{-1}}&=2\N{v_1 v_{2}^{-1}}=2\N{y_1\dots y_m z_{s}^{-1}\dots z_{1}^{-1}}\\
        &=\left(\sum_{i=1}^{m-1}\N{y_i y_{i+1}}\right) + \N{y_m z_{s}^{-1}} +
        \left(\sum_{i=1}^{s-1}\N{z_i z_{i+1}}\right) + \N{z_{1}^{-1}y_1},\\
        2\N{g}&=2\N{v_1u}=2\N{y_1\dots y_m x_1 \dots x_{k}}\\
        &=\left(\sum_{i=1}^{m-1}\N{y_i y_{i+1}}\right) + \N{y_m x_1} + \left(\sum_{i=1}^{k-1} \N{x_i x_{i+1}}\right) + \N{x_k y_1}>0,\\
        2\N{h}&=2\N{v_2u}=2\N{z_1\dots z_s x_1 \dots x_{k}}\\
        &=\left(\sum_{i=1}^{s-1}\N{z_i z_{i+1}}\right) +\N{z_s x_1} + \left(\sum_{i=1}^{k-1} \N{x_i x_{i+1}}\right) + \N{x_k z_1}>0.
        \end{split}
    \end{equation*}
    After subtracting the repeating sums from both sides, the inequality $2\N{gh^{-1}}<2\N{g}+2\N{h}$ becomes equivalent to
     \begin{equation}\label{eq:inequality_in_proof}
        \N{y_m z_{s}^{-1}} + \N{z_1^{-1} y_1} < \N{y_m x_1} + \N{x_k y_1} + \N{z_s x_1} + \N{x_k z_1} + 2 \sum_{i=1}^{k-1} \N{x_i x_{i+1}}.
    \end{equation}
    Applying Lemma \ref{lem:yz} twice, we get $\N{y_m z_{s}^{-1}}<\N{y_m x_1}+\N{z_s x_1}$ and $\N{z_1^{-1} y_1}<\N{x_k y_1}+\N{x_k z_1}$, from which \eqref{eq:inequality_in_proof} follows.   
    Now we conclude from \eqref{eq:lem_second} that $2\N{w}=2\N{gh}=2\N{g}+2\N{h}$, which equals the right-hand side of \eqref{eq:main} for $w=gh=y_1\dots y_m x_1\dots x_k z_1\dots z_s x_1\dots x_k$.

    \smallskip

    Subcase 2b: $v_2=1\ne v_1=y_1\dots y_m$ and $v_1$ is cyclically reduced. Let us show that $g:=v_1u$ and $h:=u$ satisfy the assumptions of \eqref{eq:lem_second}.
    By the induction hypothesis,
    \begin{equation*}
        \begin{split}
        2\N{gh^{-1}}&=2\N{v_1}=2\N{y_1\dots y_m}
        =\left(\sum_{i=1}^{m-1}\N{y_i y_{i+1}}\right) + \N{y_m y_1},\\
        2\N{g}&=2\N{v_1u}=2\N{y_1\dots y_m x_1 \dots x_{k}}\\
        &=\left(\sum_{i=1}^{m-1}\N{y_i y_{i+1}}\right) + \N{y_m x_1} + \left(\sum_{i=1}^{k-1} \N{x_i x_{i+1}}\right) + \N{x_k y_1}>0,\\
        2\N{h}&=2\N{u}=2\N{x_1 \dots x_{k}}
        =\left(\sum_{i=1}^{k-1} \N{x_i x_{i+1}}\right) + \N{x_k x_1}>0.
        \end{split}
       \end{equation*}
        The inequality $2\N{gh^{-1}}<2\N{g}+2\N{h}$ is equivalent to
        \begin{equation}\label{eq:second_ineq_in_proof}
        \N{y_m y_1} < \N{y_m x_1} + \N{x_k y_1} + \N{x_k x_1}+2\left(\sum_{i=1}^{k-1} \N{x_i x_{i+1}}\right).
        \end{equation}
        If $x_1=y_1$ or $x_k=y_m$, \eqref{eq:second_ineq_in_proof} is clearly satisfied. Assume that $x_1\ne y_1$ and $x_k\ne y_m$.
        Without loss of generality, let $y_1 = a$; then  $y_m\in\{a,b^{\pm 1}\}$ because $v_1\in\Scr$.
        Notice that $y_m\ne x_1^{-1}\ne x_k \ne y_{1}^{-1}$ since $w=v_1uu\in \Scr$.
        
        If $y_m = a$, then $x_1 = x_k \in \{b, b^{-1}\}$. Now $\N{y_m y_1}=2\N{a}$, $\N{y_m x_1}=\N{x_ky_1}=\N{ab^{\pm 1}}$ and $\N{x_kx_1}=2\N{b}$. By the assumption of the theorem, $2\N{a}<2\N{ab^{\pm 1}}+2\N{b}$, from which \eqref{eq:second_ineq_in_proof} follows.
        
        If $y_m=b^{\pm 1}$, then $x_1\in \{a^{-1},y_m\}$, $x_k\in\{a,y_{m}^{-1}\}$ and $x_k \ne x_1^{-1}$, thus $x_kx_1 \in \{ay_m, y_{m}^{-1}a^{-1}\}$; so $\N{x_kx_1}=\N{ay_m}=\N{y_ma}=\N{y_my_1}$ and \eqref{eq:second_ineq_in_proof} holds as well.
        
        We conclude from \eqref{eq:lem_second} that $2\N{w}=2\N{gh}=2\N{g}+2\N{h}$, which equals the right-hand side of \eqref{eq:main} for $w=gh=y_1\dots y_m x_1\dots x_k x_1\dots x_k$.

        \smallskip

        Subcase 2c: $v_2=1\ne v_1=y_1\dots y_m$ and $v_1$ is not cyclically reduced. Then $v_1$ can be written as 
        \[v_1 = y_1 \dots y_p y_{p+1} \dots y_{p+q} y_{p}^{-1} \dots y_{1}^{-1},\]
        where $q\ge 1$ and $y_{p+1}\dots y_{p+q}$ is cyclically reduced. Let $g:=v_1u$ and $h:=u$. By the inductive assumption and (A1), we obtain
        \begin{equation*}
            \begin{split}
            2\N{gh^{-1}}&=2\N{v_1}=2\N{y_{p+1}\dots y_{p+q}}
            =\left(\sum_{i=p+1}^{p+q-1}\N{y_i y_{i+1}}\right) + \N{y_{p+q} y_{p+1}},\\
            \end{split}
        \end{equation*}
        while $2\N{g}$, $2\N{h}$ evaluate exactly as in Subcase 2b.

        From \eqref{eq:yz} with $x:=y_p$, $y:=y_{p+q}^{-1}$, $z:=y_{p+1}$, we get
        \[\N{y_{p+q}y_{p+1}}<\N{y_p y_{p+q}^{-1}}+\N{y_py_{p+1}}=\N{y_p y_{p+1}}+\N{y_{p+q}y_{p}^{-1}}.\]
        Therefore,
        \[2\N{gh^{-1}}=\left(\sum_{i=p+1}^{p+q-1}\N{y_i y_{i+1}}\right) + \N{y_{p+q} y_{p+1}}<\sum_{i=p}^{p+q}\N{y_i y_{i+1}}<2\N{g}+2\N{h}.\]
        As before, an application of \eqref{eq:lem_second} yields \eqref{eq:main} for $w=v_1uu$.

        \smallskip

        Subcase 2d: $v_1=v_2=1$, $w=u^2=x_1\dots x_k x_1\dots x_k$. By the induction hypothesis,
        \[2\N{u}=\left(\sum_{i=1}^{n-1}\N{x_i x_{i+1}}\right) + \N{x_n x_1}.\]
        Using \eqref{eq:g^n}, we get
        \[2\N{w}=2\N{u^2}=4\N{u}=2\left(\sum_{i=1}^{n-1}\N{x_i x_{i+1}}\right) + 2\N{x_k x_1},\]
        which is exactly \eqref{eq:main} for $w$.
\end{proof}

The following corollary is a reformulation of \eqref{eq:main} for $w$ written as a product of ``syllables''.

\begin{cor}\label{cor:syllables}
    Let $\N{\blank}$ be a pseudo-length on a group $G$ and $a,b\in G$.
    If $(a,b)$ is a ping-pong pair, then for $w= a^{m_1}b^{n_1}\dots a^{m_k}b^{n_k}$, where $k\ge 1$ and $m_1,\dots,m_k,n_1,\dots,n_k\in \mathbb{Z}\setminus\{0\}$, we have:
    \begin{equation}\label{eq:syllables}
    \N{w}=\N{a}\sum_{i=1}^{k}(\lvert m_i\rvert - 1) + \N{b} \sum_{i=1}^{k} (\lvert n_i\rvert - 1) + \frac{N}{2}\N{ab^{-1}} + \frac{2k -N}{2} \N{ab},
    \end{equation}
    where $N$ denotes the number of sign changes in the sequence $(m_1,n_1,\dots,m_k,n_k,m_1)$, i.e., 
    $N:=\lvert\{1\le i\le k \colon m_i n_i <0\}\rvert + \lvert\{i\le i\le k\colon n_i m_{i+1}<0\}\rvert$, where $m_{k+1}:=m_{1}$.
\end{cor}

\begin{proof}
    Notice that for $x,y\in \{a,b\}$, $x\ne y$, $\varepsilon,\eta\in\{-1,1\}$, we have $\N{x^{\varepsilon}y^{\eta}}=\N{ab}$ if $\sgn{\varepsilon}=\sgn{\eta}$, and $\N{x^{\varepsilon}y^{\eta}}=\N{ab^{-1}}$ if $\sgn{\varepsilon}\ne \sgn{\eta}$. Expanding the right-hand side of \eqref{eq:main} using the above relations and $\N{x^2}=2\N{x}$ for $x\in\Sigma$, we get $2\N{w}=2r$, where $r$ denotes the right-hand side of \eqref{eq:syllables}; note that $r$ is a well-defined element of $\Lambda$ because $N$ is an even number.
\end{proof}

\begin{rem}
    In the special case when $\N{ab}=\N{ab^{-1}}>\N{a}+\N{b}$, the formula \eqref{eq:syllables} becomes
    $\N{w}=k\N{ab} +\N{a}\sum_{i=1}^{k}(\lvert m_i\rvert - 1) + \N{b} \sum_{i=1}^{k} (\lvert n_i\rvert - 1)$, which is exactly \cite[Lemma 1.2]{Parry}.
    On the other hand, if $\N{ab}=\N{a}+\N{b}\ge \N{ab^{-1}}$, we get 
    $\N{w}=\N{a}\sum_{i=1}^{k}\lvert m_i\rvert + \N{b}\sum_{i=1}^{k}\lvert n_i\rvert - \frac{N}{2} (\N{ab}-\N{ab^{-1}})$. This case was considered in \cite[Lemma 1.6]{Parry}. Parry showed there the inequality $\N{w}\le \N{a}\sum_{i=1}^{k}\lvert m_i\rvert + \N{b}\sum_{i=1}^{k}\lvert n_i\rvert$ and noted that the equality holds if $m_1,\dots,m_k,n_1,\dots,n_k$ have the same sign (i.e., $N=0$). He also proved that if the equality holds and $N>0$, then $\N{ab^{-1}}=\N{a}+\N{b}$. We have proved the converse implication. Moreover, our formula covers the missing case when $N >0$ and $\N{ab^{-1}}<\N{a}+\N{b}=\N{ab}$, giving the precise value of $\N{w}$.
\end{rem}

Our next result says that, under certain conditions, a formula similar to \eqref{eq:main} well defines a pseudo-length on $F(a,b)$. The proof is rather long and divided into several cases. In some of them much effort is needed to show that the axiom (A0) is satisfied. 

\begin{prop}\label{prop:existence}
    Let $\alpha,\beta,\gamma,\delta \in \Lambda$ be such that
    \begin{gather}
        \gamma - \alpha - \beta \in 2\Lambda,\quad \delta - \alpha - \beta \in 2\Lambda;\label{eq:2L}\\
        \text{either } \gamma=\delta > \alpha +\beta  \quad \text{or } \max\{\gamma, \delta\}=\alpha+\beta;\label{eq:either_or}\\
        \alpha>0, \;\beta>0,\; \lvert\alpha-\beta\rvert < \min\{\gamma,\delta\}.\label{eq:good_pair_alfa_beta}
    \end{gather}
    Define $f\colon \Sigma\times \Sigma\to \Lambda$ as follows:
    \begin{equation}\label{eq:cond_on_f}
        \begin{array}{c}
            f(a,a)=2\alpha, \, f(b,b)=2\beta, \, f(a,b)=\gamma, \, f(a,b^{-1})=\delta,\\
        f(x,y)=f(y,x)=f(y^{-1},x^{-1}), \, f(x,x^{-1})=0 \quad \text{for all } x,y\in \Sigma.
        \end{array}
    \end{equation}
    Let $\N{\blank}\colon F(a,b)\to \Lambda_{+}$ be defined by $\N{1}:=0$ and, for $g\ne 1$:
    \begin{equation}\label{eq:w}
        \N{g}:=\frac{1}{2} \left(\sum_{i=1}^{n-1}f(x_i, x_{i+1}) + f(x_n, x_1)\right),
    \end{equation}
    where $x_1 \dots x_n \in \Scr$, $x_i\in\Sigma$ $(1\le i\le n)$, is any cyclically reduced word representing an element of $F(a,b)$ conjugated to $g$. Then $\N{\blank}$ is a pseudo-length on $F(a,b)$.
\end{prop}

\begin{proof}
    Define a function $\phi\colon \Sigma^{*}\to \Lambda$ by putting $\phi(1):=0$ and
    \begin{equation}\label{eq:phi}
   \phi(w):=\left(\sum_{i=1}^{n-1}f(x_i, x_{i+1})\right) + f(x_n, x_1) \quad \text{for } w=x_1\dots x_n \in \Sigma^{*},
    \end{equation}
    where $n\ge 1$ and $x_i\in \Sigma$ $(1\le i\le n)$. Observe that $\phi$ is invariant with respect to any cyclic shift of $w$. Moreover, it follows from \eqref{eq:cond_on_f} that $\phi(w^{-1})=\phi(w)$ for all $w\in \Sigma^{*}$.

    Assume that $w$ is a reduced word. We claim that $\phi(w)\in 2\Lambda$. Indeed, if $w=a^{n}$ for some $n\in \mathbb{Z}$, then $\phi(w)=\lvert n\rvert f(a,a)=2\lvert n\rvert \alpha$; similarly $\phi(b^n)=2\lvert n\rvert \beta$. If $w$ is cyclically reduced, we may assume that $w= a^{m_1}b^{n_1}\dots a^{m_k}b^{n_k}$, where $k\ge 1$ and $m_1,\dots,m_k,n_1,\dots,n_k\in \mathbb{Z}\setminus\{0\}$. Then, by expanding the right-hand side of \eqref{eq:phi} as in Corollary \ref{cor:syllables}, we obtain
    \[\phi(w)=2\alpha\sum_{i=1}^{k}(\lvert m_i\rvert - 1) + 2\beta \sum_{i=1}^{k} (\lvert n_i\rvert - 1) + N\delta + (2k-N) \gamma\]
    for an even number $N$, so $\phi(w)\in 2\Lambda$. If $w$ is not cyclically reduced, we may assume without loss of generality that $w=a^{m_1}b^{n_1}\dots a^{m_k}b^{n_k}a^{m_{k+1}}$, where $k\ge 1$, $m_1,\dots,m_k,m_{k+1},n_1,\dots,n_k\in \mathbb{Z}\setminus\{0\}$ and $\sgn{m_{k+1}}\ne \sgn{m_1}$. We then have
    \[
    \begin{split}
    \phi(w)=2\alpha\sum_{i=1}^{k+1}(\lvert m_i\rvert - 1) + 2\beta \sum_{i=1}^{k} (\lvert n_i\rvert - 1) + N\delta + (2k-N) \gamma \\
    = 2\alpha\sum_{i=1}^{k+1}(\lvert m_i\rvert - 1) + 2\beta \sum_{i=1}^{k} (\lvert n_i\rvert - 1) + 2k\gamma + N(\delta - \gamma),
    \end{split}\]
    where $N$ denotes the number of sign changes in the sequence $(m_1,n_1,\dots,m_k,n_k,m_{k+1})$. Since $\delta - \gamma \in 2\Lambda$ by \eqref{eq:2L}, we get $\phi(w)\in 2\Lambda$.

    So far we have shown that the formula \eqref{eq:w} well defines a function from $F(a,b)$ to $\Lambda_{+}$, and this function satisfies the axioms (A1), (A4) and (A5). Moreover, $\N{g}=0$ if and only if $g=1$. Since (A3) implies (A2) for $g,h$ with $\N{g}>0$, $\N{h}>0$, and (A2) holds trivially if $g=1$ or $h=1$, we only need to show that $\N{\blank}$ satisfies (A3) and (A0).

    It follows from \eqref{eq:either_or}, \eqref{eq:good_pair_alfa_beta} and \eqref{eq:cond_on_f} that $f$ satisfies the assumptions of Lemma \ref{lem:technical}.
    Let us make an important observation. If $u_1,u_2\in \Sigma^{*}$ and $x,y,z\in \Sigma$, then
    $\phi(u_1 yx^{-1}xz u_2)-\phi(u_1 y z u_2)=f(y,x^{-1})+f(x,z)-f(y,z)$. By \eqref{eq:technical} with $y^{-1}$ in the place of $y$ and \eqref{eq:cond_on_f}, $\phi(u_1 yx^{-1}xz u_2)-\phi(u_1 y z u_2)\ge 0$.
    Similarly $\phi(y x^{-1} x)-\phi(y)=f(y,x^{-1})+f(x,y)-f(y,y)\ge 0$ and $\phi(x^{-1}x)-\phi(1)=0$.
    Therefore, any reduction in $w\in \Sigma^{*}$ does not increase the value of $\phi$. Since $\phi$ is invariant under cyclic shifts, the same holds for a cyclic reduction of $w\in\Sigma^{*}$.

    We proceed to the proof of (A3) and (A0) for $g,h \in F(a,b)\setminus\{1\}$. Note that if the second alternative in (A3) holds, (A0) becomes trivial. To show (A0), we will often make use of the fact that $\gamma-\delta,\gamma+\delta \in 2\Lambda$, which is a consequence of \eqref{eq:2L}.

    In what follows, given $u,v\in \Sr$, we will write $u\cdot v$ for the reduced word obtained from $uv$ by performing as much cancellation as possible.
    We identify $F(a,b)$ with $\Sr$ canonically, so that $u\cdot v$ represents the product of $u$ and $v$ in $F(a,b)$.

    Assume that $g,h\in \Sr\setminus\{1\}$. We will consider several cases.

    \medskip

    Case 1: both $g$ and $h$ are cyclically reduced. Let $g=x_1 \dots x_n$, $n\ge 1$, $x_i\in \Sigma$ $(1\le i\le n)$, and $h=y_1 \dots y_m$, $m\ge 1$, $y_i\in \Sigma$ $(1\le i\le m)$.
    Using \eqref{eq:phi}, we calculate
    \begin{equation}\label{eq:phi_gh}
        \begin{split}
        \phi(gh)&=\phi(x_1 \dots x_n y_1 \dots y_m)
        =\phi(x_1\dots x_n)-f(x_n,x_1)+f(x_n,y_1)\\
        &+\phi(y_1\dots y_m)-f(y_m,y_1)+f(y_m,x_1)\\
        &=\phi(g)+\phi(h)-f(x_n,x_1)+f(x_n,y_1)-f(y_m,y_1)+f(y_m,x_1)
        \end{split}
    \end{equation}
    and similarly
    \begin{align}
        \phi(gh^{-1}) &= \phi(x_1 \dots x_n y_{m}^{-1} \dots y_{1}^{-1})= \phi(x_1 \dots x_n) - f(x_n, x_1) + f(x_n, y_{m}^{-1}) \notag \\
        &+ \phi(y_{m}^{-1} \dots y_{1}^{-1}) - f(y_{1}^{-1}, y_{m}^{-1}) + f(y_{1}^{-1}, x_1) \label{eq:phi_gh_-1} \\
        &= \phi(g) + \phi(h) - f(x_n, x_1) + f(x_n, y_{m}^{-1}) - f(y_m, y_1) + f(y_{1}^{-1}, x_1) \notag
    \end{align}

    Subcase 1a: $x_1=y_1$. Then 
    $\phi(gh)=\phi(g)+\phi(h)$ by \eqref{eq:phi_gh} and $gh$ is cyclically reduced, so
    $\N{g\cdot h}=\frac{1}{2} \phi(gh)=\frac{1}{2}(\phi(g)+\phi(h))=\N{g}+\N{h}$.
    We also have $\phi(gh^{-1})=\phi(g)+\phi(h)-f(x_n, x_1)+f(x_n, y_{m}^{-1})-f(y_m, x_1)$ by \eqref{eq:phi_gh_-1}.
    Since $f(x_n, y_{m}^{-1})\le f(x_n, x_1)+f(y_m, x_1)$ by Lemma \ref{lem:technical}, we get
    $\N{g\cdot h^{-1}}\le \frac{1}{2}\phi(gh^{-1})\le \frac{1}{2}(\phi(g)+\phi(h))=\N{g}+\N{h}$. Thus, we obtain $\max\{\N{g\cdot h},\N{g\cdot h^{-1}}\}=\N{g}+\N{h}$, as desired.

    \smallskip

    Subcase 1b: $x_n=y_m$. Then $g^{-1}$ and $h^{-1}$ have the same first letter and the situation can be reduced to Subcase 1a.

    \smallskip

    Subcase 1c: $x_1=y_{m}^{-1}$ or $x_n=y_{1}^{-1}$. Then we can apply Subcase 1a or 1b with $h^{-1}$ in the place of $h$.
    
    \smallskip

    After Subcases 1a--1c (and using that $g,h\in\Scr$), we may assume that the elements $x_1, x_{n}^{-1}, y_1, y_{m}^{-1}$ are pairwise distinct, and so they exhaust the entire $\Sigma$. It follows that both $gh$ and $gh^{-1}$ are cyclically reduced. We distinguish three remaining subcases according to which element of $\{x_{n}^{-1}, y_{m}^{-1}, y_1\}$ is equal to $x_{1}^{-1}$.

    \smallskip

    Subcase 1d: $x_{1}^{-1}=x_{n}^{-1}$, so $x_1=x_n$ and $y_1=y_m$. Assume without loss of generality that $x_1=a$, $y_1=b$. Then
    $\phi(gh)=\phi(g)+\phi(h)-f(a,a)+f(a,b)-f(b,b)+f(b,a)=\phi(g)+\phi(h)+2\gamma-2\alpha - 2\beta$ and
    $\phi(gh^{-1})=\phi(g)+\phi(h)+2\delta-2\alpha-2\beta$.

    If $\gamma = \delta > \alpha + \beta$, then
    $\N{g\cdot h}=\N{g\cdot h^{-1}}=\N{g}+\N{h}+\gamma-\alpha-\beta>\N{g}+\N{h}$; so the first alternative in (A3) holds. Moreover, $\N{g\cdot h}-\N{g}-\N{h}=\gamma-\alpha-\beta\in 2\Lambda$ by \eqref{eq:2L}, hence (A0) is satisfied.

    If $\max\{\gamma,\delta\}=\alpha + \beta$, then $\max\{\N{g\cdot h},\N{g\cdot h^{-1}}\}=\N{g}+\N{h}$.

    \smallskip

    Subcase 1e: $x_{1}^{-1}=y_{m}^{-1}$, so $x_1=y_m$ and $y_1=x_n$. Assume without loss of generality that $x_1=a$, $y_1=b$.
    Then $\phi(gh)=\phi(g)-f(b,a)+f(b,b)+\phi(h)-f(a,b)+f(a,a)=\phi(g)+\phi(h)+2\alpha+2\beta-2\gamma$ and
    $\phi(gh^{-1})=\phi(g)-f(b,a)+f(b,a^{-1})+\phi(h)-f(b,a)+f(b^{-1},a)=\phi(g)+\phi(h)+2\delta-2\gamma$.
    If $\gamma = \delta > \alpha + \beta$, then $\max\{\N{g\cdot h},\N{g\cdot h^{-1}}\}=\N{g\cdot h^{-1}}=\N{g}+\N{h}$.
    If $\max\{\gamma,\delta\}=\gamma=\alpha + \beta$, then $\max\{\N{g\cdot h},\N{g\cdot h^{-1}}\}=\N{g\cdot h}=\N{g}+\N{h}$. If $\max\{\gamma,\delta\}=\delta=\alpha + \beta>\gamma$, then $\N{g\cdot h}=\N{g\cdot h^{-1}}=\N{g}+\N{h}+\alpha + \beta - \gamma > \N{g}+\N{h}$, hence (A3) holds. Moreover, $\N{g\cdot h}-\N{g}-\N{h} = -(\gamma - \alpha -\beta)\in 2\Lambda$ by \eqref{eq:2L}.

    \smallskip
    
    Subcase 1f: $x_{1}^{-1}=y_1$, so $x_n=y_{m}^{-1}$. Then we can apply Subcase 1e with $h^{-1}$ in the place of $h$.

    \bigskip

    Before we split the proof into other cases, we prove a useful fact.

    \smallskip

    \noindent {\bf Claim.} Let $g',h',w\in \Sigma^{*} \setminus\{1\}$ and $W_{+}:= g'wh'w^{-1}$, $W_{-}:=g'wh'^{-1}w^{-1}$. If $g',h',W_{+},W_{-}\in \Scr$, then
    \begin{gather}
        \N{W_{+}}=\N{W_{-}}>\N{g'}+\N{h'},\label{eq:claim_1}\\ 
        \N{W_{+}}-\N{g'}-\N{h'}\in 2\Lambda.\label{eq:claim_2}
    \end{gather} 

    \smallskip

    \noindent {\em Proof of the Claim.} Let $g'=x_1\dots x_n$, $x_i\in\Sigma$ $(1\le i\le n)$, $h'=y_1\dots y_m$, $y_i\in\Sigma$ $(1\le i\le m)$, and $w=w_1 \dots w_r$, $w_i\in\Sigma$ $(1\le i\le r)$. We have $\N{W_+}=\frac{1}{2}\phi(W_+)$, $\N{W_{-}}=\frac{1}{2}\phi(W_{-})$, $\N{g'}=\frac{1}{2}\phi(g')$, $\N{h'}=\frac{1}{2}\phi(h')$.

    Let us calculate
    \begin{equation}
        \begin{split}
            \phi(W_{+})&=\phi(g')-f(x_n,x_1)+f(x_n,w_1)\\&+\phi(w)-f(w_r, w_1)
            +f(w_r,y_1)\\
            &+\phi(h')-f(y_m,y_1)+f(y_m,w_{r}^{-1})\\
            &+\phi(w^{-1})-f(w_{1}^{-1}, w_{r}^{-1})+f(w_{1}^{-1},x_1),
        \end{split}
    \end{equation}
    \begin{equation}
        \begin{split}
            \phi(W_{-})&=\phi(g')-f(x_n,x_1)+f(x_n,w_1)\\&+\phi(w)-f(w_r, w_1)
            +f(w_r,y_{m}^{-1})\\
            &+\phi(h'^{-1})-f(y_{1}^{-1},y_{m}^{-1})+f(y_{1}^{-1},w_{r}^{-1})\\
            &+\phi(w^{-1})-f(w_{1}^{-1}, w_{r}^{-1})+f(w_{1}^{-1},x_1).
        \end{split}
    \end{equation}
    It follows from the properties of $f$ and $\phi$ that $\phi(W_{+})=\phi(W_{-})$. Define the following sums:
    \begin{gather*}
    S_1:= f(x_n,w_1)+f(w_{1}^{-1},x_1)-f(x_n,x_1)=f(w_{1}^{-1},x_{n}^{-1})+f(w_{1}^{-1},x_1)-f(x_n,x_1),\\
    S_2:= f(y_m,w_{r}^{-1})+f(w_{r},y_1)-f(y_m,y_1)=f(w_r,y_{m}^{-1})+f(w_r,y_1)-f(y_m,y_1).
    \end{gather*}
    We have $S_1>0$ and $S_2>0$ by Lemma \ref{lem:technical}. Moreover,
    \begin{gather*}
    \N{W_+}-\N{g'}-\N{h'}=\N{W_{-}}-\N{g'}-\N{h'}\\
    =\frac12 \left(\phi(W_{+})-\phi(g')-\phi(h')\right)=\phi(w)+\frac12 \left(S_1+S_2\right)-f(w_r,w_1).
    \end{gather*}
    Since $S_1>0$, $S_2>0$ and $\phi(w)\ge f(w_r,w_1)$, we obtain \eqref{eq:claim_1}. Since $\phi(w)\in 2\Lambda$, the condition \eqref{eq:claim_2} is equivalent to 
    \begin{equation}\label{eq:S_1+S_2}
    S_1+S_2 - 2 f(w_r,w_1)\in 4\Lambda.
    \end{equation}
    Note that, since $x_{n}^{-1}\ne w_1 \ne x_1$ and $x_{n}^{-1}\ne x_1$, either $x_1=x_n$ and $w_1\in \Sigma\setminus\{x_1,x_{1}^{-1}\}$, or $x_1$, $x_n$ are powers of distinct letters in $\{a,b\}$ and $w_1\in \{x_{1}^{-1},x_{n}\}$. Similarly, either $y_1=y_m$ and $w_r\in \Sigma\setminus\{y_1,y_{1}^{-1}\}$, or $y_1$, $y_m$ are powers of distinct letters and $w_r\in \{y_1,y_{m}^{-1}\}$.

    \smallskip

    Suppose that $x_1=x_n=:x$ and $y_1=y_m=:y$. Without loss of generality, we assume that $x\in\{a,a^{-1}\}$.

    If $y\in\{a,a^{-1}\}$, then $w_1,w_r\in \{b,b^{-1}\}$, so
    $S_1=f(w_{1}^{-1},a)+f(w_{1}^{-1},a^{-1})-f(a,a)=\gamma + \delta - 2 \alpha$ and similarly $S_2=\gamma + \delta - 2\alpha$.
    Hence we have $S_1+S_2=2(\gamma + \delta)-4\alpha \in 4\Lambda$. Since $f(w_r,w_1)\in \{0, f(b,b)\}=\{0,2\beta\}$, we also have $2f(w_r,w_1)\in 4\Lambda$, which proves \eqref{eq:S_1+S_2}.

    If $y\in \{b,b^{-1}\}$, then $w_1\in \{b,b^{-1}\}$ and $w_r\in \{a,a^{-1}\}$, so
    $S_1=f(w_{1}^{-1},a)+f(w_{1}^{-1},a^{-1})-f(a,a)=\gamma + \delta - 2\alpha$
    and $S_2=f(w_r,b)+f(w_r,b^{-1})-f(b,b)=\gamma + \delta - 2\beta$. Moreover, $f(w_r,w_1)\in \{\gamma,\delta\}$.
    Hence, $S_1 + S_2 - 2f(w_r,w_1)$ equals $2\gamma - 2\alpha - 2\beta\in 4\Lambda$ or $2\delta - 2\alpha - 2\beta\in 4\Lambda$, thus \eqref{eq:S_1+S_2} is true.

    \smallskip

    Suppose now that $x_1\ne x_n$ and $y_1\ne y_m$. Then $w_1\in \{x_{1}^{-1},x_n\}$ and $w_r\in\{y_1,y_{m}^{-1}\}$. Hence,
    \begin{gather*}
        S_1 = f(w_{1}^{-1},w_{1}^{-1})+f(x_1,x_{n}^{-1})-f(x_1,x_n)=f(w_{1}^{-1},w_{1}^{-1})\pm(\gamma-\delta),\\
        S_2 = f(w_r,w_r)+f(y_1,y_{m}^{-1})-f(y_1,y_m)=f(w_r,w_r)\pm(\gamma-\delta).
    \end{gather*}
    If $w_{1}^{-1}=w_r\in\{a,a^{-1}\}$, then $S_1 + S_2 - 2f(w_r,w_1)=4\alpha + c(\gamma-\delta)$ with $c\in \{-2,0,2\}$, so the sum in \eqref{eq:S_1+S_2} belongs to $4\Lambda$. If $w_{1}^{-1}=w_r\in\{b,b^{-1}\}$, the argument is similar.

    If $w_1=w_r$, we may assume without loss of generality that $w_1=a=w_r$. Then $S_1 + S_2 - 2f(w_1,w_r)=4\alpha + c(\gamma-\delta)-4\alpha=c(\gamma-\delta)\in 4\Lambda$.

    If $w_1\in\{a,a^{-1}\}$ and $w_r\in\{b,b^{-1}\}$, then
    $S_1 + S_2 = 2 \alpha + 2\beta + c(\gamma - \delta)$ and $2f(w_r,w_1)\in \{2\gamma,2\delta\}$. Hence, the sum \eqref{eq:S_1+S_2} equals $2\alpha + 2\beta - 2\gamma + c(\gamma-\delta)\in 4\Lambda$ or $2\alpha + 2\beta - 2\delta + c(\gamma-\delta)\in 4\Lambda$.

    \smallskip

    Finally, suppose without loss of generality that $x_1=x_n=:x\in\{a,a^{-1}\}$, $y_1\ne y_m$.
    Then $w_1\in \{b,b^{-1}\}$, $w_r\in \{y_1,y_{m}^{-1}\}$ and 
    $S_1=\gamma + \delta - 2\alpha$, $S_2=f(w_r,w_r)\pm(\gamma-\delta)$.

    If $w_r\in \{a,a^{-1}\}$, then 
    $S_1+S_2=\gamma + \delta \pm(\gamma-\delta)\in \{2\gamma,2\delta\}$ and $2f(w_r,w_1)\in \{2\gamma,2\delta\}$, hence the sum in \eqref{eq:S_1+S_2} equals $c(\gamma-\delta)$ with $c\in\{-2,0,2\}$, so it belongs to $4\Lambda$.

    If $w_r\in \{b,b^{-1}\}$, then
    $S_1+S_2 \in \{2\gamma -2\alpha + 2\beta, 2\delta - 2\alpha + 2\beta\}$, which is a subset of $4\Lambda$ because, for example, $2\gamma - 2\alpha + 2\beta=2(\gamma-\alpha-\beta)+4\beta$. Since $2f(w_r,w_1)\in\{0,4\beta\} \subseteq 4\Lambda$, \eqref{eq:S_1+S_2} holds as well.

    \medskip

    We have thus proved the Claim.

    \bigskip

    Assume that at least one of $g,h\in \Sr\setminus\{1\}$ is not cyclically reduced.
    Then we can write
    \begin{align*}
    g=u g' u^{-1}\quad &\text{for } g'=x_1\dots x_n\in\Scr\setminus\{1\}, \ u=u_1\dots u_k\in \Sr,\\
    h=v h' v^{-1}\quad &\text{for } h'=y_1\dots y_m\in\Scr\setminus\{1\},\ 
    v=v_1\dots v_s\in\Sr,
    \end{align*}
    where $x_i,y_i,u_i,v_i\in \Sigma$ for all $i$ in the appropriate ranges and at least one of the words $u,v$ is nonempty. Notice that $\N{g}=\N{g'}$ and $\N{h}=\N{h'}$.
    Denote by $p$ the length of the longest common initial subword of $u$ and $v$.

    \medskip

    Case 2: $p<\min\{k,s\}$. Define $w:= u^{-1}\cdot v = u_{k}^{-1} \dots u_{p+1}^{-1} v_{p+1}\dots v_s$. Then $g\cdot h=ug' w h' v^{-1}$ is conjugated to $W_{+}:=g' w h' w^{-1}$,
    which is cyclically reduced. Similarly, $g \cdot h^{-1}$ is conjugated to the cyclically reduced word
    $W_{-}:=g' w h'^{-1} w^{-1}$. Since $\N{g\cdot h}=\N{W_{+}}$ and $\N{g\cdot h^{-1}}=\N{W_{-}}$, it follows from the Claim that $g$, $h$ satisfy (A3) and (A0).

    \medskip

    Case 3: $p=\min\{k,s\}$. Assume without loss of generality that $p=k$.
    
    \smallskip

    Subcase 3a: $p=s$. Then $u=v$ and
    $g\cdot h$, $g\cdot h^{-1}$ are conjugated to $g'\cdot h'$ and $g'\cdot h'^{-1}$, respectively. Since $g'$, $h'$ are cyclically reduced and $\N{g\cdot h}=\N{g'\cdot h'}$, $\N{g\cdot h^{-1}}=\N{g'\cdot h'^{-1}}$, the situation can be reduced to Case 1.

    \smallskip

    Before we split the remaining part of the proof further into subcases, assume that $p<s$.

    \smallskip

    Subcase 3b:  $x_1 \ne v_{p+1}$ and $x_n \ne v_{p+1}^{-1}$. Let $w:=u^{-1}\cdot v=v_{p+1}\dots v_s$. Then $W_{+}:=g'wh'w^{-1}$ and $W_{-}:=g'wh'^{-1}w^{-1}$ are both  cyclically reduced and conjugated to $g\cdot h$ and $g\cdot h^{-1}$, respectively. As in Case 2, the Claim yields (A3) and (A0) for $g,h$.

    \smallskip
    
    Subcase 3c: $x_1=v_{p+1}$.
    For convenience, let us extend the notation $x_i$ to $i>n$ treating the indices modulo $n$, that is, $x_{n+1}:=x_1$, etc. Define $q:=\max\{1\le i\le s-p \colon x_i = v_{p+i}\}$.

    If $q<s-p$, we can write
    \[g\cdot h=u x_1 \dots x_n x_{n+1}\dots x_{n+q} v_{p+q+1} \dots v_s y_1 \dots y_m v_{s}^{-1} \dots v_{p+q+1}^{-1} x_{q}^{-1}\dots x_{1}^{-1}u^{-1},\] which
    is conjugated to the cyclically reduced word
    $W_{+}:=g'' w h' w^{-1}$, where $g'':=x_{q+1}\dots x_{q+n}$ and $w:=v_{p+q+1}\dots v_s$.
    Similarly, $g\cdot h^{-1}$ is conjugated to the cyclically reduced word ${W}_{-}:=g'' w h'^{-1} w^{-1}$. Note that $g''$ is a cyclic shift of $g'$, so $g''\in\Scr$. Applying the Claim to $g''$, $h'$, $w$ and using the equalities $\N{g\cdot h}=\N{W_{+}}$, $\N{g\cdot h^{-1}}=\N{W_{-}}$, $\N{g}=\N{g'}=\N{g''}$, $\N{h}=\N{h'}$, we get (A3) and (A0) for $g$, $h$.

    If $q=s-p$, then $g\cdot h$, $g\cdot h^{-1}$ are conjugated to $g''\cdot h'$ and $g''\cdot h'^{-1}$, respectively. We return to Case 1 with $g''$ and $h'$ instead of $g$ and $h$.

    \smallskip

    Subcase 3d: $x_n=v_{p+1}^{-1}$. We claim that Subcase 3c applies to the pair of elements $g^{-1}=ug'^{-1}u^{-1}$, $h^{-1}=vh'^{-1}v^{-1}$. Indeed, $u$ and $v$ remain unchanged and the first letter of $g'^{-1}$ is $x_{n}^{-1}$, hence equal to $v_{p+1}$. Therefore, the conditions (A3) and (A0) hold for $g^{-1}$, $h^{-1}$, and hence (in view of (A1) and (A5)) for $g$, $h$ as well.
\end{proof}

We can now combine Theorem \ref{th:main} with Proposition \ref{prop:existence}, which leads to the following result on the existence and uniqueness of a pseudo-length on $F(a,b)$ satisfying certain conditions.

\begin{thm}\label{thm:unique}
    Assume that $\alpha,\beta,\gamma,\delta\in\Lambda$ satisfy the conditions \eqref{eq:2L}, \eqref{eq:either_or}, and \eqref{eq:good_pair_alfa_beta}. There exists exactly one pseudo-length $\N{\blank}\colon F(a,b) \to \Lambda_{+}$ such that $\N{a}=\alpha$, $\N{b}=\beta$, $\N{ab}=\gamma$, and $\N{ab^{-1}}=\delta$.
\end{thm}

\begin{proof}
    By Proposition \ref{prop:existence}, the formula \eqref{eq:w} defines a pseudo-length $\N{\blank}$ on $F(a,b)$ such that 
    \[
    \begin{array}{cc}
        \N{a}=\frac12 f(a,a)=\alpha,& \N{b}=\frac12 f(b,b)=\beta,\\
        \N{ab}=\frac12 \left(f(a,b)+f(b,a)\right) = \gamma,&
        \N{ab^{-1}}=\frac12 \left(f(a,b^{-1})+f(b^{-1},a)\right)=\delta.
    \end{array}
    \]
    
    Suppose that $\N{\blank}_{1}$, $\N{\blank}_2$ are two pseudo-lengths on $F(a,b)$ taking the values $\alpha$, $\beta$, $\gamma$, $\delta$ at $a$, $b$, $ab$, $ab^{-1}$, respectively. By elementary properties of any pseudo-length, $\N{\blank}_1$, $\N{\blank}_2$ must agree on all elements of $F(a,b)$ represented by two-letter words in $\Sr$. Hence, by Theorem \ref{th:main}, they agree on all elements represented by words in $\Scr$; so they are equal by (A1).
\end{proof}

\section{Applications}

From our results we draw the following corollary about properly discontinuous actions. It combines \cite[Lemmas 3.3.6 and 3.3.8]{Chiswell}.

\begin{cor}\label{cor:general}
    Let $G$ be a group acting by isometries on a $\Lambda$-tree $(X,d)$ with the translation length function $\N{\blank}\colon G\to \Lambda_{+}$.
    If $(a,b)\in G\times G$ is a ping-pong pair with respect to $\N{\blank}$, then the subgroup $\left< a, b\right>$ is free of rank two and acts freely, without inversions, and properly discontinuously on $(X,d)$. 

    Moreover, if $G$ is a topological group and $\N{\blank}$ is continuous at the identity, then $\left< a, b\right>$ is discrete with respect to the topology inherited from $G$.
\end{cor}

\begin{proof}
    It follows from \eqref{eq:main} that $\left< a, b\right>$ is free of rank two and $\N{\blank}$ restricted to $\left< a, b\right>$ is purely hyperbolic. Hence, the action of $\left< a, b\right>$ on $(X,d)$ is free and without inversions.

    Let ${C:=\min\big\{\N{a},\N{b},\N{ab},\N{ab^{-1}}\}>0}$. By \eqref{eq:syllables}, \eqref{eq:lem_first} and (A1), we have
    $\N{g}\ge C$ for all $g\in \left< a, b\right>\setminus\{1\}$.
    We consider two cases. If there exists $\min\{\lambda\in\Lambda\colon \lambda>0\}$, the topology of $(X,d)$ is discrete and the free action of $\left<a,b\right>$ on $X$ is clearly properly discontinuous. Otherwise, there exists $r>0$ such that $2r<C$. Suppose that $x\in X$,
    $y\in B(x,r)$, $g\in \left<a,b\right>\setminus\{1\}$ and $d(gy,x)<r$. Then $\N{g}\le d(gy,y)<2r<C$, a contradiction. Hence, the neighborhood $U:=B(x,r)$ of $x$ satisfies $gU\cap U=\emptyset$ for $g\in\left<a,b\right>\setminus\{1\}$; so the action of $\left<a,b\right>$ is properly discontinuous.

    Assume that $\N{\blank}\colon G \to \Lambda_{+}$ is continuous at $1\in G$.
    As we have shown,
    $\{1\}=\{g\in \left<a, b\right> \colon \N{g}<C\}$. The latter set is a neighborhood of $1$ in $\left<a, b\right>$. Therefore, $\{1\}$ is open in $\left<a, b\right>$ and the subgroup $\left<a, b\right>$ is discrete.
\end{proof}

Our next corollary generalizes the result \cite[Corollary 3.6]{Conder}, which was formulated for a continuous action on a $\mathbb{Z}$-tree. It was proved there by geometrical arguments and using a version of the ping-pong lemma \cite[Lemma 3.3]{Conder}. A natural situation where the corollary can be used is $G$ being the group of all isometries of a $\Lambda$-tree $(X,d)$ with the topology of pointwise convergence. It follows from \cite[Ch. X, \S3.5, p. 30]{Bourbaki} that this topology makes $G$ a topological group. We also notice that, for any $\Lambda$-metric space $(X,d)$, $d$ is continuous as a function $d\colon X\times X \to \Lambda$.

\begin{cor}\label{cor:Conder}
   Assume that $G$ is a topological group acting on a $\Lambda$-tree $(X,d)$ by isometries in such a way that for some $x_0\in X$ the map ${G\ni g\mapsto gx_0 \in X}$ is continuous. Let $\N{\blank}\colon G\to \Lambda_{+}$ be the translation length function of this action. If $(a,b)\in G\times G$ is a ping-pong pair with respect to $\N{\blank}$, then the subgroup $\left<a, b\right>$ is free of rank two and discrete.
\end{cor}

\begin{proof}
    By the assumption $d(x_0,gx_0)\to 0$ as $g\to 1$.
    It follows from \eqref{eq:t_l} that $0\le \N{g}\le d(x_0,gx_0)$ for all $g\in G$. Hence, $\N{g}\to 0=\N{1}$ as $g\to 1$; so the function $\N{\blank}$ is continuous at $1\in G$ and we can apply Corollary \ref{cor:general}.
\end{proof}

Now, we will consider purely hyperbolic pseudo-lengths on the free group $F(a,b)$ in the case when $\Lambda$ is an Archimedean totally ordered abelian group. Without loss of generality, we can assume that $\Lambda$ is a subgroup of the additive group $\mathbb{R}$.

We will show that it is always possible to find a basis of $F(a,b)$ that is a ping-pong pair. To this end, we will use an algorithm that performs Nielsen transformations (for the terminology, see \cite[p. 5]{Lyndon-Schupp}) on the pair $(a,b)$ until a ping-pong pair is obtained. The idea is not new, a similar procedure is contained in \cite[\S4]{Culler-Vogtmann} and called the {\em division process}. It is also presented in the proof of \cite[Theorem 1]{Chiswell-article}, where the terminology and notation are of a geometric nature. The proof of termination of the algorithm relies on the completeness property of $\mathbb{R}$. 

For the reader's convenience, we will present the algorithm in detail and prove its correctness, using only the properties of pseudo-lengths.

\begin{algorithm}[H]
    \caption{}
\begin{algorithmic}[1]
    \Require{a purely hyperbolic pseudo-length $\N{\blank}\colon F(a,b)\to \mathbb{R}_{+}$}
    \Ensure{a ping-pong pair generating $F(a,b)$}
    \State $(g,h):= (a,b)$
    \If{$\N{g}<\N{h}$}
        \State $(g,h):=(h,g)$
    \EndIf
    \If{$\N{gh}<\N{gh^{-1}}$}
        \State $(g,h):=(g,h^{-1})$
    \EndIf
    \While{$\N{g}-\N{h}=\N{gh^{-1}}$}
        \State $(g,h):=(gh^{-1},h)$
        \If{$\N{g}<\N{h}$}
        \State $(g,h):=(h,g)$
        \EndIf
        \If{$\N{gh}<\N{gh^{-1}}$}
        \State $(g,h):=(g,h^{-1})$
        \EndIf
    \EndWhile
    \If{$\N{g}-\N{h}<\N{gh^{-1}}$}
        \State \Return $(g,h)$
        \Else
        \State \Return $(gh^{-1},h)$
    \EndIf
\end{algorithmic}
\end{algorithm}

\begin{lem}\label{lem:algorithm}
    Let $\N{\blank}\colon G\to \Lambda_{+}$ be a pseudo-length on a group $G$. Assume that $g,h\in G$ are such that $\N{g}\ge \N{h}>0$ and $\N{gh}\ge \N{gh^{-1}}>0$.
    \begin{enumerate}
        \item If $\N{g}-\N{h}=\N{gh^{-1}}$, then $\N{[g,h]}\le 2\N{g}$.
        \item If $\N{g}-\N{h}>\N{gh^{-1}}$, then $(gh^{-1},h)$ is a ping-pong pair.
    \end{enumerate}
\end{lem}

\begin{proof}
    Assume that $\N{g}-\N{h}=\N{gh^{-1}}$. It follows from (A3) that $\max\{\N{gh},\N{gh^{-1}}\}=\N{g}+\N{h}$, so $\N{gh}=\N{g}+\N{h}$. Let us apply (A3) to the pair $(gh,g^{-1}h)$ of hyperbolic elements of $G$. If the first alternative in (A3) holds, then
    \[\N{ghg^{-1}h}=\N{g^2}=2\N{g}>\N{gh}+\N{g^{-1}h}=2\N{g},\]
    a contradiction. Thus, $\N{ghg^{-1}h}\le \max\{\N{ghg^{-1}h}, \N{g^2}\}=2\N{g}$. Let us apply now (A2) to the pair $(ghg^{-1},h)$. We obtain $\N{ghg^{-1}h}=\N{ghg^{-1}h^{-1}}$ or
    \[\max\{\N{ghg^{-1}h},\N{ghg^{-1}h^{-1}}\}\le \N{ghg^{-1}}+\N{h}=2\N{h}\le 2\N{g}.\]
    In either case, $\N{[g,h]}=\N{ghg^{-1}h^{-1}}\le 2\N{g}$.

    Assume that $\N{g}-\N{h}>\N{gh^{-1}}$. It follows from (A3) applied to the pair $(gh^{-1},h)$ that either
    $\N{g}=\N{gh^{-2}}>\N{gh^{-1}}+\N{h}$ or
    $\max\{\N{g},\N{gh^{-2}}\}=\N{gh^{-1}}+\N{h}$. The latter equality would lead to a contradiction. Hence, we have
    $\min\{\N{(gh^{-1})h},\N{(gh^{-1})h^{-1}}\}=\N{g}>\lvert \N{gh^{-1}}-\N{h}\rvert$, so $(gh^{-1},h)$ is a ping-pong pair.
\end{proof}

\begin{prop}\label{prop:termination}
    Algorithm 1 terminates after a finite number of steps and returns the correct output.
\end{prop}

\begin{proof}
    First, observe that each of the assignments in the lines 3, 5, 7, 9, 11 of Algorithm 1 is a Nielsen transformation, so $(g,h)$ is always a basis of $F(a,b)$. Since $\N{\blank}$ is purely hyperbolic, we thus have $\N{g}>0$, $\N{h}>0$.
    Moreover, it can be routinely checked that $\N{[g,h]}=\N{[h,g]}=\N{[g,h^{-1}]}=\N{[gh^{-1},h]}$, so $\N{[g,h]}=\N{[a,b]}>0$ remains constant during the execution of the algorithm.

    Suppose that the algorithm does not terminate. This means that, during its execution, we never exit the {\bf while} loop in the lines 6--11. Let us define a sequence $(g_n,h_n)_{n\ge 0}\subseteq G\times G$ as follows.
    Let $(g_0,h_0)$ denote the values of the variables $g$ and $h$ when entering the {\bf while} loop for the first time.
    For $n\ge 1$, let $(g_n,h_n)$ denote the values of $g$ and $h$ at the end of the $n$-th execution of the loop.
    The instructions in the lines 2--5 and 8--11 normalize $(g,h)$ so that
    \[\N{g_n}\ge \N{h_n}>0 \quad \text{and} \quad \N{g_nh_n}\ge \N{g_n h_{n}^{-1}}>0 \quad \text{for all } n\ge 0.\]
    Let us define a sequence of real numbers $\Delta_n:= 2\N{g_n} + 2\N{h_n} - \N{[g_n,h_n]}$ for $n\ge 0$. It follows from Lemma \ref{lem:algorithm} (i) that $\Delta_n>0$ for all $n\ge 0$.
    The assignment in the line 7 diminishes $\Delta_n$ by $2\N{h_n}$ and those from the lines 9, 11 do not change it, so \begin{equation}\label{eq:Delta}
    \Delta_n = \Delta_{n-1} - 2\N{h_n} < \Delta_{n-1} \quad \text{for all } n\ge 1.
    \end{equation}
    Similarly, we can see that $(\N{g_n})_{n\ge 0}$, $(\N{h_n})_{n\ge 0}$ are nonincreasing sequences of positive real numbers.
    Let $\bar{g}:=\lim_{n\to\infty} \N{g_n}$, $\bar{h}:=\lim_{n\to\infty} \N{h_n}$ and $\bar{\Delta}:=\lim_{n\to\infty} \Delta_{n}$. We have $0\le \bar{h}\le \bar{g}$, $\bar{\Delta} \ge 0$, and by \eqref{eq:Delta} we obtain $\bar{h}=0$.

    Suppose that $\bar{g}>0$, then there exists $n_0$ such that $2\N{h_n}<\bar{g}\le \N{g_n}$ for $n\ge n_0$. It follows that $\N{g_n h_{n}^{-1}}=\N{g_n}-\N{h_n}>\N{h_n}$ for $n\ge n_0$. Hence, the line 9 is no longer executed starting with the $(n_0+1)$-th iteration of the {\bf while} loop, so $\N{h_{n+1}}=\N{h_n}$ for $n\ge n_0$.
    Therefore, $\N{h_n}$ is constant for $n\ge n_0$ and $\bar{h}=\N{h_{n_0}}>0$, a contradiction.
    Thus, $\bar{g}=\bar{h}=0$ and we obtain $\bar{\Delta}=2\bar{g}+2\bar{h}-\lim_{n\to \infty }\N{[g_n,h_n]} = -\N{[a,b]}<0$, which is another contradiction.
    
    We have shown that the algorithm must exit the {\bf while} loop. Hence, we eventually have $\N{g}-\N{h}\ne \N{gh^{-1}}$. If the condition in the line 12 is true, $(g,h)$ is clearly a ping-pong pair. Otherwise, $(gh^{-1},h)$ is a ping-pong pair by Lemma \ref{lem:algorithm} (ii).
\end{proof}

\begin{rem}
 Conder presented a related algorithm \cite[Algorithm 5.2]{Conder}, which determines whether or not the group generated by two isometries $a$, $b$ of a locally finite $\mathbb{Z}$-tree $X$ is both free of rank two and discrete (with respect to the pointwise convergence topology). In detail, the algorithm performs Nielsen transformations on the pair $(a,b)$ until either a ping-pong pair is obtained or one of the isometries is elliptic. The proof of termination of the algorithm \cite[Theorem 4.2]{Conder} relies on the well-ordering of the positive integers.
\end{rem}

Notice that, if $\N{\blank}\colon G\to \Lambda_{+}$ is a purely hyperbolic pseudo-length on $G$ and $\sigma$ is an automorphism of $G$, then the mapping $G\ni g \mapsto \N{\sigma(g)}\in \Lambda_{+}$ is also a purely hyperbolic pseudo-length. Thus we have a right action of the group $\Aut{(G)}$ of automorphisms of $G$ on the set $\PPsi{(G)}$ of all purely hyperbolic pseudo-lengths on $G$. The axiom (A1) guarantees that the group $\Inn{(G)}$ of inner automorphism of $G$ acts trivially on $\PPsi{(G)}$; so there is an action of $\Out{(G)}:=\Aut{(G)}/\Inn{(G)}$, the group of outer automorphisms of $G$, on $\PPsi{(G)}$.

Culler and Vogtmann \cite{CV1986} introduced an important space on which $\Out{(F_n)}$ acts ($F_n$ is free of rank $n$); it is called the {\em outer space} and can be defined as a space of minimal free actions of $F_n$ on simplicial $\mathbb{R}$-trees. Two actions are identified if their translation length functions differ only by a scalar factor. The outer space is then topologized as a subspace of the projective space $\mathbb{P}^{\mathcal{C}}:=(\mathbb{R}^{\mathcal{C}}\setminus\{0\})/\mathbb{R}^{\times}$, where $\mathcal{C}$ denotes the set of conjugacy classes in $F_n$. Since any purely hyperbolic pseudo-length on $F_2$ is the translation length function of a minimal free action on a simplicial $\mathbb{R}$-tree \cite[Theorems 3.4.2\,(c) and 5.2.6]{Chiswell}, the outer space in rank two can be thought of as the projectivization of $\PPsi{(F_2)}$. In \cite[\S6]{Culler-Vogtmann} a finite-dimensional embedding of this space (and its closure) was constructed.

We will find a subset $Y'\subseteq \PPsi{(F_2)}$ that has nonempty intersection with every $\Aut{(F_2)}$-orbit in $\PPsi{(F_2)}$. This set, after projectivization, can be embedded in $\mathbb{R}^3$. Moreover, we believe that $Y'$ contains exactly one point of each of the orbits. If this be true, we would have a nice description of the $\Aut{(F_2)}$-orbits in the Culler--Vogtmann outer space of $F_2$. 

\bigskip

Assume that $\Lambda$ is a nontrivial subgroup of $\mathbb{R}$ and $\alpha,\beta,\gamma,\delta\in\Lambda$ satisfy the conditions \eqref{eq:2L}, \eqref{eq:either_or} and \eqref{eq:good_pair_alfa_beta}. We will denote by $\N{\blank}_{\alpha,\beta,\gamma,\delta}$ the unique pseudo-length on $F(a,b)$ as in Theorem \ref{thm:unique}.

\begin{thm}\label{thm:description_of_orbits}
    Let $\{0\}\ne \Lambda \le \mathbb{R}$ and $\N{\blank}\colon F(a,b)\to \Lambda_{+}$ be a purely hyperbolic pseudo-length. There exists an automorphism $\sigma$ of $F(a,b)$ and $\alpha,\beta,\gamma,\delta\in\Lambda$ satisfying \eqref{eq:2L}, \eqref{eq:either_or}, and \eqref{eq:good_pair_alfa_beta}, such that
    \[\N{w}=\N{\sigma(w)}_{\alpha,\beta,\gamma,\delta} \quad \text{for all } w\in F(a,b).\]
\end{thm}

\begin{proof}
    Let us execute Algorithm 1 and denote by $(g,h)$ the ping-pong pair obtained as the output. Define $\alpha:=\N{g}$, $\beta:=\N{h}$, $\gamma:=\N{gh}$, $\delta:=\N{gh^{-1}}$.
    The condition \eqref{eq:either_or} follows from (A3), and \eqref{eq:good_pair_alfa_beta} is true since $(g,h)$ is a ping-pong pair. If $\gamma=\delta > \alpha+\beta$, \eqref{eq:2L} is a consequence of (A0). Assume without loss of generality that $\gamma=\max\{\gamma,\delta\}=\alpha+\beta$. We will show that $\delta - \alpha - \beta \in 2\Lambda$. By \eqref{eq:main} we have
    $\N{ghg^{-1}h}=\N{gh}+\N{gh^{-1}}=\gamma + \delta$.
    Hence,
    \[\N{ghg^{-1}h}-\N{ghg^{-1}}-\N{h}=\gamma + \delta - 2\beta = \delta + \alpha - \beta>0\]
    It now follows from (A0) that $\delta + \alpha - \beta \in 2\Lambda$, so $\delta - \alpha - \beta\in 2\Lambda$ as well.

    Let $\sigma$ be the automorphism of $F(a,b)$ sending $g$ to $a$ and $h$ to $b$. Clearly, the function $F(a,b)\ni w\mapsto \N{w}_{1}:=\N{\sigma^{-1}(w)}\in \Lambda_{+}$ is a pseudo-length on $F(a,b)$, and
    $\N{a}_1=\alpha$, $\N{b}_1=\beta$, $\N{ab}_1=\gamma$, $\N{ab^{-1}}_1=\delta$. We deduce from Theorem \ref{thm:unique} that $\N{\blank}_1=\N{\blank}_{\alpha,\beta,\gamma,\delta}$, so
    $\N{w}=\N{\sigma(w)}_1=\N{\sigma(w)}_{\alpha,\beta,\gamma,\delta}$ for all $w\in F(a,b)$.
\end{proof}

\begin{rem}\label{rem:final_rem}
    If $\Lambda=\mathbb{R}$, Theorem \ref{thm:description_of_orbits} can be expressed as follows:

    Every $\Aut{(F_2)}$-orbit in $\PPsi{(F_2)}$ contains an element of the set
    \[
    Y:=\{\N{\blank}_{\alpha,\beta,\gamma,\delta}: \alpha,\beta,\gamma,\delta\in\mathbb{R} \; \text{satisfy \eqref{eq:either_or} and \eqref{eq:good_pair_alfa_beta}}\}.
    \]

    We can find a smaller set with this property, namely, let $Y':=Y_1 \cup Y_2$, where
    \begin{equation}\label{eq:Y1_and_Y2}
         \begin{aligned}
        Y_1&:=\{\N{\blank}_{\alpha,\beta,\alpha+\beta,\delta}: \alpha + \beta > \delta\ge \alpha \ge \beta >0\},\\
        Y_2&:=\{\N{\blank}_{\alpha,\beta,\gamma,\gamma}: \alpha\ge \beta >0,\; \gamma\ge \alpha+\beta\}.
        \end{aligned}
    \end{equation}

    We will briefly show that every element of $Y$ can be mapped into $Y'$ by an automorphism of $F_2$. Note that we can always ensure that $\alpha\ge \beta$ and $\gamma\ge \delta$; just apply the automorphisms $a\mapsto b, b \mapsto a$ and $a\mapsto a, b \mapsto b^{-1}$ if necessary.
    
    Suppose that $\N{\blank}_{\alpha,\beta,\gamma,\delta}\in Y$ does not belong to the $\Aut{(F_2)}$-orbit of any element of $Y_2$. Then we can assume without loss of generality that $\gamma=\alpha + \beta > \delta > \alpha - \beta \ge 0$. If $\delta \ge \alpha$, then $\N{\blank}_{\alpha,\beta,\gamma,\delta}\in Y_1$. Otherwise, let $\sigma\in\Aut{(F_2)}$ be given by $a\mapsto ab^{-1}, b\mapsto b^{-1}$ and denote by $\alpha',\beta',\gamma',\delta'$ the values that $\N{\sigma(\blank)}_{\alpha,\beta,\gamma,\delta}$ takes at $a,b,ab,ab^{-1}$, respectively.
    We have $\alpha'=\delta$, $\beta'=\beta$, $\gamma'=\N{ab^{-2}}_{\alpha,\beta,\gamma,\delta} = \delta + \beta$, $\delta'=\alpha$. Notice that $\delta' \ge \max\{\alpha',\beta'\}$ and $\gamma'=\alpha'+\beta'=\delta + \beta>\alpha =\delta'$. Swapping $a$ with $b$ if necessary, we ensure that $\alpha'\ge \beta'$ (without changing $\gamma'$, $\delta'$). Thus, we get an element of $Y_1$.
    
    In addition, note that Algorithm 1 can be extended to include (as the final step, if necessary) the ``normalization'' procedure described above. The modified algorithm would then output a ping-pong pair $(g, h)$ such that $\alpha:=\N{g}$, $\beta:=\N{h}$, $\gamma:=\N{gh}$, and $\delta:=\N{gh^{-1}}$ satisfy the conditions appearing in the definition of one of the sets in \eqref{eq:Y1_and_Y2}. For convenience, we could also include $\alpha,\beta,\gamma,\delta$ as part of the output. After this extension, when given a pseudo-length $\N{\blank}_{\alpha,\beta,\gamma,\delta}\in Y'$ as input, the algorithm would clearly return $(a,b)$ along with the (unchanged) numbers $\alpha,\beta,\gamma,\delta$.
\end{rem}

One can see that the sets \eqref{eq:Y1_and_Y2}, after projectivization (setting $\gamma:=1$), correspond to two perpendicular triangles in the $\alpha,\beta,\delta$-space. Their closures intersect at the segment: $1/2 \le \alpha\le 1$, $\beta=1-\alpha$, $\delta=1$.

We have some tentative computational evidence suggesting that the set $Y'$ contains exactly one element of every $\Aut{(F_2)}$-orbit in $\PPsi{(F_2)}$. Let us therefore state the following conjecture.

\begin{conj}\label{conj:conjecture}
    Let $\Lambda=\mathbb{R}$ and $Y':=Y_1\cup Y_2$, where $Y_1,Y_2 \subseteq \PPsi{(F_2)}$ are given by \eqref{eq:Y1_and_Y2}. Then every $\Aut{(F_2)}$-orbit in  $\PPsi{(F_2)}$ contains exactly one element of $Y'$.
\end{conj}

To prove this conjecture, it would be sufficient to show that if ${\N{\blank} \in \PPsi{(F_2)}}$ and $\sigma$ is an elementary Nielsen automorphism of $F_2$, then Algorithm 1, modified as described in the final paragraph of Remark \ref{rem:final_rem}, returns the same numbers $\alpha,\beta,\gamma,\delta$ for $\N{\blank}$ and $\N{\sigma({\blank})}$.

\newpage

\bibliographystyle{siam}
\bibliography{references}

\end{document}